\pdfoutput=1
\documentclass[english,11pt]{amsart}
\usepackage{amsopn,amsthm,amsfonts,amsmath,amssymb,amscd}
\usepackage{babel}
\usepackage{amssymb,tikz}

\newtheorem{Theorem}{Theorem}[section]

\newtheorem{Example}[Theorem]{Example}
\newtheorem{Remark}[Theorem]{Remark}

\newenvironment{Proof*}{{\it Proof.}}

\newcommand{\NN}{\mathbb{N}}

\newcommand{\ZZ}{\mathbb{Z}}

\newcommand{\BB}{\mathcal{B}}

\newcommand{\diag}{{\rm diag}}

\newcommand{\diam}[1]{{\rm diam}(#1)}
\newcommand{\girth}[1]{{\rm girth}(#1)}

\begin{document}

\title{The unitary Cayley graph of a semiring}
\author{David Dol\v zan}
\date{\today}

\address{D.~Dol\v zan:~Department of Mathematics, Faculty of Mathematics
and Physics, University of Ljubljana, Jadranska 21, SI-1000 Ljubljana, Slovenia; e-mail: 
david.dolzan@fmf.uni-lj.si}

 \subjclass[2010]{}
  \thanks{The author acknowledges the financial support from the Slovenian Research Agency  (research core funding No. P1-0222)}

\bigskip

\begin{abstract} 
We study the unitary Cayley graph of a matrix semiring. We find bounds for its diameter, clique number and independence number, and determine its girth. We also find the relationship between the diameter and the clique number of a unitary Cayley graph of a semiring $S$ and a matrix semiring over $S$.
\end{abstract}

\maketitle 

\section{Introduction}

\bigskip

The study of graphs that are associated to different algebraic structures has been one of the most important and active areas in algebraic combinatorics during the past few years.  Among these graphs, there are a few that are of special importance, none more than the commuting graphs, zero-divisor graphs and (unitary) Cayley graphs. 

Unitary Cayley graphs have been studied as objects of independent interest (see \cite{berri, dejter, klotz}) but are of particular relevance in the study of graph representations. These studies begun in \cite{erdos} and continued in many other papers. In \cite{akhtar} the authors studied the unitary Cayley graph associated to a finite ring, determining its diameter, girth, eigenvalues, etc. This research area has remained very active, some recent interesting results for example include the study of unitary Cayley graphs of matrix rings (\cite{chen, ratta}) and generalized unitary Cayley graphs of finite rings (\cite{chelvam}).

As far as the author of this paper is aware, the unitary Cayley graph has not yet been studied in a semiring setting and this is the topic of this paper. A \emph{semiring} is a set $S$ equipped with binary operations $+$ and $\cdot$ such that $(S,+)$ is a commutative monoid 
with identity element 0 and $(S,\cdot)$ is a monoid with identity element 1. In addition, operations $+$ and $\cdot$
are connected by distributivity and 0 annihilates $S$. A semiring is \emph{commutative} if $ab=ba$ for all $a,b \in S$.

The theory of semirings has many possible applications in optimization theory, automatic control, models of discrete event networks and graph theory (see e.g. \cite{baccelli, cuninghame, li2014, zhao2010}).  For an extensive theory of semirings, we refer the reader to \cite{hebisch}. There are many natural examples of commutative semirings, for example, the set of nonnegative integers (or reals) with the usual operations of addition and multiplication. Other examples include distributive lattices, tropical semirings, dio\"{\i}ds, fuzzy algebras, inclines and bottleneck algebras. 
A semiring $S$ is called \emph{entire} if $ab=0$ for some $a, b \in S$ implies $a=0$ or $b=0$ and it is called \emph{antinegative} or \emph{zero-sum-free}, if $a+b=0$ for some $a, b \in S$ implies that $a=b=0$. Antinegative semirings are also called \emph{antirings}. The simplest example of an antinegative semiring  is the binary Boolean semiring $\BB$, the set $\{0,1\}$ in which addition and multiplication are the same as in $\ZZ$ except that $1+1=1$.

This paper is organized as follows. In the next section, we define some basic notions that we shall need throughout the paper. The third section is the main section of this paper, where we study the properties of unitary graphs of matrix semirings. In doing this, we rely on the fact that the structure of the group of units is well understood over some matrix semirings. Firstly, we study the diameter and girth of the unitary Cayley graph of a (matrix) semiring.
The diameter of a graph is an often studied problem in connection with different graphs prescribed to algebraic structures.

The diameter of the Cayley graph was first studied in 1988 in the case of symmetric groups (see \cite{babai}). Recently, the diameter of the unitary Cayley graph of a  ring was studied in \cite{akhtar} and \cite{su}. Also, the girth of the Cayley graph has been studied, for example in \cite{wrikat} in the case of dihedral groups, as well as many other settings.
Our main results include Theorem \ref{diammatS}, where we find the relationship between the diameter of a unitary Cayley graph of semiring $S$ and the diameter of a unitary Cayley graph of the $k$ by $k$ matrix semiring over $S$, and Theorem \ref{girth}, where we determine the girth of a unitary Cayley graph of a matrix semiring.
Finally, we tackle the somewhat related problems of the clique and independence numbers. While these are quite easy to calculate in the case of unitary Cayley graphs over finite rings (see \cite[Proposition 6.1]{akhtar}), it turns out that over semirings, the situation is somewhat more complicated.
The main results here include Theorem \ref{cn}, where we prove that the clique number of a unitary Cayley graph of an entire additively cancellative antiring $S$ is equal to the clique number of a $k$ by $k$ matrix semiring over $S$, and Theorem \ref{independence}, where we find the bounds for the independence number of a unitary Cayley graph of a matrix semiring.

\bigskip


\bigskip
\section{Preliminaries}
\bigskip

For a semiring $S$, we denote by $S^*$ the group of invertible elements in $S$.
We denote by $\Gamma(S)$ the \emph{unitary Cayley graph} of $S$. 
The vertex set $V(\Gamma(S))$ of $\Gamma(S)$ is the set of elements in $S$. 
Now, in the unitary Cayley graph of a ring, there is an edge 
between two distinct vertices $x$ and $y$ if $x-y \in S^*$ (which is of course, equivalent to the condition that $y-x \in S^*$). 
Since in general, we do not have subtraction in $S$,
we extend the above definition in a natural way and say that  an unordered pair of vertices $x,y \in V(\Gamma(S))$, $x \neq y$, is an edge 
$x \sim y$ in $\Gamma(S)$ if there exists $u \in S^*$ such that $x+u=y$ or $y+u=x$.

We shall need some further graph theoretical definitions. The sequence of edges $x_0 \sim x_1$, $ x_1 \sim x_2$, ..., $x_{k-1} \sim x_{k}$ in a graph is called \emph{a path of length $k$}. We shall denote this path by $x_0 \sim x_1 \sim x_2 \sim \ldots \sim x_k$. 
The \emph{distance} between vertices $x$ and $y$ is the length of the shortest
path between them, denoted by $d(x,y)$. If there is no path between $x$ and $y$, we define $d(x,y)=\infty$.
The \emph{diameter}, denoted $\diam{\Gamma}$, of the graph $\Gamma$ is the supremum of all distances between any two vertices of the graph.
The \emph{girth}, denoted $\girth{\Gamma}$, of the graph $\Gamma$ is the length of the shortest cycle in $\Gamma$. If there are no cycles in $\Gamma$, we say that $\girth{\Gamma}=\infty$.  The clique number of a graph $\Gamma$, denoted $\omega(\Gamma)$, is the number of vertices in a maximum clique of $\Gamma$. An independent set is a set of vertices in a graph, no two of which are adjacent, and the independence number of a graph $\Gamma$, denoted $\alpha(\Gamma)$, is the number of vertices in a maximum independent set of $\Gamma$.

We say that $x \in S$ is \emph{additively cancellative} if for any $z, y \in S$ such that $x+y=x+z$ we have $y=z$. Semiring $S$ is  \emph{additively cancellative} if every element of $S$ is additively cancellative.
A set $\{a_1,a_2,\ldots,a_r\} \subseteq S$ of nonzero elements is called an \emph{orthogonal decomposition of 1 } in $S$ if $a_1+a_2+\ldots+a_r=1$ and $a_ia_j=0$ for all $i \ne j$. 
For a semiring $S$, we denote by $M_n(S)$ the semiring of all $n$ by $n$ matrices with entries in $S$. We shall denote by $E_{ij} \in M_n(S)$ the matrix with $1$ at entry $(i,j)$ and zeros elsewhere.  
For a matrix $A \in M_n(S)$, we shall denote by $A_{ij} \in S$ the $(i,j)$-th entry of $A$.  
Furthermore, let $N_n$ denote the set $\{1,2,\ldots,n\}$ and let $S_n$ denote the symmetric group on the set $N_n$. For a permutation $\sigma \in S_n$, we shall denote by $P_\sigma$ the permutation matrix corresponding to permutation $\sigma$, thus $(P_\sigma)_{ij}=1$ if and only if $\sigma(i)=j$, and $0$ otherwise. Finally, we shall denote the diagonal matrix in $M_n(S)$ with elements $a_1,a_2,\ldots,a_n \in S$ along the diagonal by $\diag(a_1,a_2,\ldots,a_n)$.

The following theorem will be an essential tool in several proofs in this paper, so we state it explicitely.

\bigskip

\begin{Theorem}\cite[Theorem 1]{dolobl}
\label{invertible}
 If $S$ is a commutative antiring, then $A \in M_{n}(S)$ is invertible if and only if 
  $$A=D \sum\limits_{\sigma \in S_n}{a_\sigma P_\sigma}\, ,$$
 where $D$ is an invertible diagonal matrix, $P_{\sigma}$ is a permutation matrix and 
 $\sum_{\sigma \in S_n} a_\sigma=1$ is an orthogonal decomposition of $1$.
\end{Theorem}

\bigskip


\bigskip
\section{The properties of the unitary Cayley graph of a semiring}
\bigskip

In the first part of this section we shall study the diameter of the unitary Cayley graph of a semiring. The situation in the semiring setting turns out to be quite distinct from the one in rings. Let us start with a few illustrative examples.

\bigskip
\begin{Example}
\label{ex0}
Let $S=\NN_0$, the semiring of non-negative integers. Observe that $S^*=\{1\}$ and that $\Gamma(S)$ is a path graph, so it is connected, but $\diam{\Gamma(S)}=\infty$.
\end{Example}

\begin{Example}
\label{ex1}
Let $n$ be an integer and $S=N_n \cup \{0\}$, with the operations defined as $a\oplus b=\min\{a+b,n\}$ and $a\odot b=\min\{ab,n\}$. Observe that $(S, \oplus, \odot)$ is a semiring with $S^*=\{1\}$, so $d(0,n)=n$ and therefore $\diam{\Gamma(S)}=n$.
\end{Example}

\begin{Example}
\label{ex2}
Let $S=\BB[x]/(x^2)$. Observe that $S$ is a semiring with $S^*=\{1\}$, but $d(1,x)=\infty$, so $\Gamma(S)$ is disconnected and thus also $\diam{\Gamma(S)}=\infty$.
\end{Example}

\bigskip

Compare the above examples with the diameters of the unitary Cayley graphs over rings. If $R$ is a commutative Artinian ring, then $\diam{\Gamma(R)} \in \{1,2,3,\infty\}$ (see \cite[Theorem 3.1]{akhtar}), where $\diam{\Gamma(R)} = \infty$ if and only if $R$ is a direct product of local rings such that at least two of those local rings have their residue fields isomorphic to $GF(2)$. It turns out that the unitary Cayley graph of a semiring has an altogether more complicated structure. For example, the unitary Cayley graph of a commutative ring is always regular (\cite[Proposition 2.2]{akhtar}), while none of the graphs in Examples \ref{ex0} and \ref{ex1} are.

Therefore, we limit ourselves here to two special cases. We have the following theorem.

\bigskip

\begin{Theorem}
\label{diamS}
Let $S$ be a semiring such that $\Gamma(S)$ is a connected graph. Then
\begin{enumerate}
\item
If there exist positive integers $m < n$ such that $m \cdot 1 = n \cdot 1$ in $S$, then $\diam{\Gamma(S)} \leq 2(n-1)|S^*|$.
\item
If $S^*$ is closed under addition, then $\diam{\Gamma(S)} \leq 2$.
\end{enumerate}
\end{Theorem}
\begin{proof}
Choose $x,y \in S$. Since $\Gamma(S)$ is connected, we have at least one path between $x$ and $y$, $x=z_0 \sim z_1 \sim z_2 \sim \ldots \sim z_t=y$. Because $z_i \sim z_{i+1}$ for any $i$ implies that there exists $u \in S^*$ such that either $z_i + u = z_{i+1}$ or $z_{i+1} + u = z_{i}$, the idea of this proof is to separate additions arising from both of these two cases. So, choose any $i \in \{0,1,\ldots,t-2\}$ such that there exist integers $j,k,l \geq 0$ with $z_{i+j}=z_i+u$ where $u$ is a sum of $j$ units,
$z_{i+j+k}+v=z_{i+j}$ where $v$ is a sum of $k$ units and $z_{i+j+k+l}=z_{i+j+k}+w$ where $w$ is a sum of $l$ units. Notice that 
$z_{i+j+k+l}+v=z_{i+j+k}+v+w=z_{i+j}+w=z_i+u+w$. Applying this argument inductively, we see that there exist an integer $r \geq 0$ and a path 
$$x=z_0' \sim z_1' \sim z_2' \sim \ldots \sim z_r' \sim z_{r+1}' \sim \ldots \sim z_t'=y,$$ 
where $z_r'=x+\alpha$ and $y+\beta=z_r'$, where $\alpha$ and $\beta$ are either zero or sums of units.
Examine now both of the two cases separately.
\begin{enumerate}
\item
If $|S^*| = \infty$, there is nothing to prove. So, assume that $|S^*| < \infty$. Choose $u \in S^*$ and observe that $mu = nu$ and therefore $au \in \{u,2u,\ldots,(n-1)u\}$ for every positive integer $a$. This implies that every sum of units in $S$ can be written as a sum of at most $(n-1)|S^*|$ elements. By the above argument, this implies that $d(x,y) \leq 2(n-1)|S^*|$. 
\item
Since a sum of units is a unit, the arguments above imply that any path between $x, y \in S$ is of the form $x=z_0' \sim z_1' = x+\alpha \sim y$, where $y+\beta=z_1'$ for some $\alpha, \beta \in S^* \cup \{0\}$, and is thus of length at most $2$, therefore $d(x,y) \leq 2$.
\end{enumerate}
\end{proof}

\bigskip

\begin{Remark}
\label{maxlength}
In the case (1) of Theorem \ref{diamS}, we actually prove that for every $x, y \in S$ there exist paths from both $x$ and $y$  to $x+\gamma=y+\gamma$, where $\gamma=\sum\limits_{u \in S^*}{(n-1)u} \in S$.  
\end{Remark}

\bigskip

As the next example shows, the bounds from Theorem \ref{diamS} can be achieved in both cases.

\bigskip

\begin{Example}
\label{bounds}
Choose an integer $r \geq 1$. Let $S=N_r \cup (N_{r-1} + x)=\{0,1,\ldots,r,x,1+x,\ldots,(r-1)+x\}$, where $x+x=x^2=x$ and $r+a=r$ for every $a \in S$. It can be easily checked that $S$ is a semiring with $|S|=2r+1$, $S^*=\{1\}$ and $r \cdot 1 = (r+1) \cdot 1$. Note that $x+(r-1)+1=x+r=r$, so $r \sim x+(r-1)$ and therefeore $0 \sim 1 \sim 2 \sim \ldots \sim r-1 \sim r \sim x+(r-1) \sim x+(r-2) \sim \ldots \sim x+1 \sim x$ is a path of length $2r$. So, $\Gamma(S)$ is connected and Theorem \ref{diamS} states that $\diam{\Gamma(S)} \leq 2r$, but obviously $\diam{\Gamma(S)} = 2r$.  
If $r=1$ then $S=\{0,1,x\}$ is a semiring with $1+1=1$, so $S^*$ is closed under addition and $\diam{\Gamma(S)} = 2$.
\end{Example}

\bigskip

Since in general, the structure of the group of units in a semiring can be quite varied, it is difficult to examine the unitary Cayley graph of an arbitrary semiring. We therefore turn our attention to the matrix semirings, where at least in some instances, the group of units is well known. We prove the following theorem.

\bigskip

\begin{Theorem}
\label{diammatS}
Let $k \geq 2$ be an integer and let $S$ be a semiring such that $\Gamma(S)$ is a connected graph. Then
\begin{enumerate}
\item
If there exist positive integers $m < n$ such that $m \cdot 1 = n \cdot 1$ in $S$, then $\diam{\Gamma(M_k(S))} \leq 2k (n-1)|S^*|$. 
\item
If $S^*$ is closed under addition, then $\diam{\Gamma(M_k(S))} \leq 2k$.
\end{enumerate}
Moreover, if $S$ in an entire antiring, then $\diam{\Gamma(M_k(S))} \geq k \, \diam{\Gamma(S)} $.
\end{Theorem}
\begin{proof}
Let us firstly prove the moreover part. If $S$ is an entire antiring, then Theorem \ref{invertible} yields that every invertible matrix in $M_k(S)$ is of the form $DP$, where $D$ is an invertible diagonal matrix and $P$ is a permutation matrix. Now, choose $a \neq b \in S$. Denote $A=a(E_{11}+E_{12}+\ldots+E_{1k})$ and $B=b(E_{11}+E_{12}+\ldots+E_{1k})$. Since adding any unit in $M_k(S)$ to $A$ or $B$ changes exactly one of the elements in the first row (by adding a unit to that particular element), we conclude that $d(A,B) \geq k \, d(a,b)$. This implies that $\diam{\Gamma(M_k(S))} \geq k \, \diam{\Gamma(S)}$.

Now, choose an invertible diagonal matrix $D \in M_k(S)$ and a permutation matrix $P$. Observe that $DPP^TD^{-1}=I$, so $DP$ is a unit in $M_k(S)$.  Also, let $\sigma=(1 2 \ldots k) \in S_k$ and note that for any $i,l \in N_k$, we have $$\sigma^l(i)=\begin{cases}
     i+l, & \text{if }  i+l \leq k;\\
     i+l-k, & \text{otherwise}.
     \end{cases}$$
This implies that for every $i, j \in N_k$, there exists exactly one $l \in N_k$ such that $\sigma^l(i)=j$ and thus $P_\sigma + P_{\sigma^2} + \ldots + P_{\sigma^k}$ is a matrix with all elements equal to $1$. Now, choose $A, B \in M_k(S)$.
\begin{enumerate}
\item
We have $d(A_{ij}, B_{ij}) \leq 2 (n-1)|S^*|$ for every $i, j \in \{1,2, \ldots, k\}$ by Theorem \ref{diamS}. By Remark \ref{maxlength}, we know that there exist $s=(n-1)|S^*|$ (not necessarily distinct) units $u_1, u_2, \ldots, u_s$ in $S$ such that $A_{ij} \sim A_{ij} + u_1 \sim A_{ij} + u_1 + u_2 \sim \ldots \sim A_{ij} + \gamma = B_{ij} + \gamma \sim \ldots \sim B_{ij} + u_1 + u_2 \sim B_{ij} + u_1 \sim B_{ij}$ for every $i, j \in \{1,2, \ldots, k\}$.  This proves that there also exists a path 
$A \sim A + u_1P_\sigma \sim A + (u_1+u_2)P_\sigma \sim \ldots \sim A + \gamma P_\sigma \sim A + \gamma P_\sigma + u_1 P_{\sigma^2} \sim  A + \gamma P_\sigma + (u_1 + u_2)P_{\sigma^2} \sim \ldots \sim A + \gamma P_\sigma + \gamma P_{\sigma^2} \sim \ldots \sim A + \gamma (P_\sigma + P_{\sigma^2} + \ldots + P_{\sigma^k})=B + \gamma (P_\sigma + P_{\sigma^2} + \ldots + P_{\sigma^k}) \sim \ldots \sim B + \gamma P_{\sigma} + \gamma P_{\sigma^2} \sim \ldots \sim B + \gamma P_{\sigma} \sim \ldots \sim B$ of length at most $2k (n-1)|S^*|$, where
$\gamma=\sum\limits_{i=1}^s{u_i} \in S$. Therefore  $\diam{\Gamma(M_k(S))} \leq 2k (n-1)|S^*|$. 

\item
We have $d(A_{ij}, B_{ij}) \leq 2$ for every $i, j \in \{1,2, \ldots, k\}$ by Theorem \ref{diamS}. Since $S^*$ is closed under addition, this implies that for every $i, j \in \{1,2,\ldots,k\}$ there exist $u_{ij}, v_{ij} \in  S^*$ such that we have a path $A_{ij} \sim A_{ij} + u_{ij} = B_{ij} + v_{ij} \sim B_{ij}$ of length at most 2. 
Denote $D_i=\diag(u_{1\sigma^i(1)},u_{2\sigma^i(2)},\ldots, u_{k\sigma^i(k)})$ and $E_i=\diag(v_{1\sigma^i(1)},v_{2\sigma^i(2)},\ldots, v_{k\sigma^i(k)})$ for every $i=1,2,\ldots,k$. Since $D_i$ and $E_i$ are invertible matrices for every $i=1,2,\ldots,k$, we have a path
$A \sim A + D_1 P_{\sigma} \sim A + D_1 P_{\sigma} + D_2 P_{\sigma^2} \sim \ldots \sim A + D_1 P_{\sigma} + D_2 P_{\sigma^2} + \ldots + D_k P_{\sigma^k} = B + E_1 P_{\sigma} + E_2 P_{\sigma^2} + \ldots + E_k P_{\sigma^k} \sim \ldots \sim B + E_1 P_{\sigma} + E_2 P_{\sigma^2} \sim B + E_1 P_{\sigma} \sim B$.
This is a path between $A$ and $B$ in $\Gamma(M_k(S))$ of length at most $2k$, so  $\diam{\Gamma(M_k(S))} \leq 2k$.

\end{enumerate}
\end{proof}

\bigskip

Again, the upper bounds can be achieved in both cases as the next example shows.

\bigskip

\begin{Example}
\label{boundsmat}
Let $r \geq 1$ and let $S=N_r \cup (N_{r-1} + x)=\{0,1,\ldots,r,x,1+x,\ldots,(r-1)+x\}$, where $x+x=x^2=x$ and $r+a=r$ for every $a \in S$, as in Example \ref{bounds}.
It is clear that $S$ is an entire antinegative semiring. Example \ref{bounds} states that $\diam{\Gamma(S)} = 2r$ and Theorem \ref{diammatS} shows that $\diam{\Gamma(M_k(S))} = 2kr$.
If $r=1$ then $S^*$ is closed under addition and
$\diam{\Gamma(M_k(S))} = 2k$.
\end{Example}

\bigskip

Next, let us also examine the girth of the unitary Cayley graph of a matrix semiring. We have the following theorem.

\bigskip

\begin{Theorem}
\label{girth}
Let $S$ be a semiring and $k \geq 2$. Then $\girth{\Gamma(M_k(S))} \leq 4$. Moreover:
\begin{enumerate}
\item
If there exist $u, v \in S^*$ such that $u+v \in S^* \setminus \{u,v\}$ then $\girth{\Gamma(M_k(S))} = 3$;
\item
otherwise, if $S$ is an entire additively cancellative antiring, then $\girth{\Gamma(M_k(S))} = 4$.
\end{enumerate}
\end{Theorem}
\begin{proof}
Denote by $P \neq I$ a permutation matrix in $M_k(S)$. Since $P$ is invertible, $P \sim 0 \sim I \sim I+P \sim P$ is a $4$-cycle in $M_k(S)$, so $\girth{\Gamma(M_k(S))} \leq 4$. If $u+v=w$ for some units $u,v, w \in S$ with $w \neq u,v$, then matrices $0, uI$ and $wI$ form a $3$-cycle in $M_k(S)$, so $\girth{\Gamma(M_k(S))} = 3$.

So, suppose now that $S$ is an entire additively cancellative antiring. Choose $u,v \in S^*$. If $u+v \in S^*$ then by our assumption $u+v=u$ or $u+v=v$. But since $S$ is additively cancellative, this is a contradiction. So, we have proved that the sum of two units is not a unit in $S$. Choose $A, U \in M_n(S)$ such that $U$ is invertible (so $A$ and $A+U$ are neighbours in the graph). Suppose that there exists $B \in M_k(S)$ such that $A, A+U$ and $B$ form a $3$-cycle. We have two possibilites: there exists an invertible $W \in M_k(S)$ such that either $B=A+W$ or $A=B+W$. Assume firstly that $B=A+W$. Since $A+U$ and $B$ are neighbours, there exists an invertible $V \in M_k(S)$ such that $A+W+V=A+U$ or $A+U+V=A+W$. Assume without loss of generality that $A+W+V=A+U$. Since $S$ is additively cancellative, we have $W+V=U$. By Theorem \ref{invertible} we know that there exist permutations $\sigma, \rho, \tau \in S_k$ and invertible diagonal matrices $D_1, D_2$ and $D_3$ that such that $W=D_1 P_\sigma$, $V=D_2 P_\rho$ and $U=D_3 P_\tau$. However, we can reason now that $\sigma=\rho=\tau$, so $W+V=(D_1+D_2)P_\sigma=U$. Since both $D_1$ and $D_2$ have units along their diagonals and we have previously proved that a sum of two units is not a unit in $S$, we see that $D_1+D_2$ is not an invertible matrix, which now yields a contradiction. In the case $A=B+W$, the reasoning is similar.
\end{proof}

\bigskip

Note that in case $S$ is not additively cancellative, it can happen that $\girth{\Gamma(M_k(S))} = 3$ (even if the semiring does not even contain two distinct units), as the following example shows.

\bigskip

\begin{Example}
\label{girth3}
Let, $A=\left[
 \begin{matrix}
 0 & 1 & 1 \\
 1 & 1 & 0 \\
 1 & 1 & 1 
 \end{matrix}
 \right]$ and $U=\left[
 \begin{matrix}
 1 & 0 & 0 \\
 0 & 0 & 1 \\
 0 & 1 & 0 
 \end{matrix}
 \right]$ in $M_3(\BB)$. Observe that $U$ is a unit in $M_3(\BB)$, so $A \sim A + I  \sim A + I + U = A + U \sim A$ is a 3-cycle in $\Gamma(M_3(\BB))$.
This shows that $\girth{\Gamma(M_3(\BB))} = 3$ and similarly we can show that $\girth{\Gamma(M_k(\BB))} = 3$ for all $k \geq 3$. Note however that a simple verification yields $\girth{\Gamma(M_2(\BB))} = 4$.
\end{Example}

\bigskip

A related notion to the notion of girth is the notion of the clique number. The next theorem describes the clique number of the unitary Cayley graph of a matrix semiring. Specifically, we have the following.

\bigskip

\begin{Theorem}
\label{cn}
Let $S$ be a semiring and $k \geq 2$. Then $\omega(\Gamma(M_k(S))) \geq \omega(\Gamma(S))$.
Moreover, if $S$ is an entire additively cancellative antiring, then $\omega(\Gamma(M_k(S)))=\omega(\Gamma(S))$. 
\end{Theorem}
\begin{proof}
If $\{w_1,w_2,\ldots,w_r\}$ is a clique in $\Gamma(S)$ then $\{w_1I,w_2I,\ldots,w_rI\}$ is a clique in $\Gamma(M_k(S))$, thus $\omega(\Gamma(M_k(S))) \geq \omega(\Gamma(S))$. 

So, suppose now that $S$ is an entire additively cancellative antiring and the set $W=\{A_1, A_2, \ldots, A_r\}$ is a maximal clique in $\Gamma(M_k(S))$. 
Let us prove that there exist a matrix $A \in M_k(S)$, a permutation matrix $P \in M_k(S)$ and invertible diagonal matrices $D_1,D_2, \ldots, D_{r-1}$ such that $D_i+D_{i+1}+\ldots+D_j$ is an invertible diagonal matrix for all $1 \leq i < j \leq r-1$ and that $W=\{A,A+D_1P,A+(D_1+D_2)P, \ldots, A+(D_1+D_2+\ldots+D_{r-1})P\}$. If $r=1$ there is nothing to prove. Suppose that $r=2$. There exists a unit $U$ in $M_k(S)$ such that $A_1=A_2+U$ or $A_2=A_1+U$. In both cases, there exists a matrix $A \in M_k(S)$ such that $W=\{A,A+U\}$. By Theorem \ref{invertible}, $U=DP$ where $D$ is an invertible diagonal matrix and $P$ is a permutation matrix, so the assertion holds. Now, assume that $r \geq 3$ and let us proceed with induction on $r$. Since $\{A_1, A_2, \ldots, A_{r-1}\}$ is a clique in $\Gamma(M_k(S))$, we have by the induction hypothesis that there exist a matrix $B \in M_k(S)$, a permutation matrix $P \in M_k(S)$ and invertible diagonal matrices $D_1,D_2, \ldots, D_{r-2}$ such that $D_i+D_{i+1}+\ldots+D_j$ is an invertible diagonal matrix for all $1 \leq i < j \leq r-2$ and $W=\{B,B+D_1P,B+(D_1+D_2)P, \ldots, B+(D_1+D_2+\ldots+D_{r-2})P, A_r\}$. Now, $A_r$ is connected to every other vertex in $W$, so there may exist some numbers $t \in \{1,2,\ldots,r-1\}$ such that $A_r+U=B+(D_1+D_2+\ldots+D_{t-1})P$ for some invertible matrix $U \in M_k(S)$. Now, if any such numbers $t$ do exist, choose the smallest one among them. On the other hand, if there exists no such $t$, define $t=r-1$. 

Assume firstly that $t=1$. So, $B=A_r+U$. Again, $U=D_0Q$ for some invertible diagonal matrix $D_0 \in M_k(S)$ and some permutation matrix $Q \in M_k(S)$ by Theorem \ref{invertible}. Since $r \geq 3$, there exists an edge between $A_r$ and $B+D_1P=A_r+D_0Q+D_1P$, so by Theorem \ref{invertible} there exist an invertible diagonal matrix $D$ and a permutation matrix $R$ such that either $A_r+DR=A_r+D_0Q+D_1P$ or $A_r+D_0Q+D_1P+DR=A_r$. Since $S$ is additively cancellative and antinegative, the latter is not possible. Thus, $A_r+DR=A_r+D_0Q+D_1P$, so $DR=D_0Q+D_1P$, since $S$ is additively cancellative. But $R$ is a permutation matrix, so the fact that $S$ is antinegative now implies that $P=Q=R$. Thus, we have proved that $\{A_r, A_r + D_0P, A_r+(D_0+D_1)P, \ldots, A_r+(D_0+D_1+\ldots+D_{r-2})P\}$ is a clique in $\Gamma(M_k(S))$.
Assume finally that $t \geq 2$. This implies that 
$B+(D_1+D_2+\ldots+D_{t-2})P+V=A_r$ for some invertible matrix $V \in M_k(S)$. By Theorem \ref{invertible}, we can again conclude that $V=D_0P$ for some invertible diagonal matrix $D_0$. Denote $E_i=D_i$ for $i=1,2,\ldots,t-2$, $E_{t-1}=D_0$ and $E_i=D_{i-1}$ for $i=t,t+1,\ldots,r-1$. Observe that
 $\{B, B + E_1P, B+(E_1+E_2)P, \ldots, B+(E_1+E_2+\ldots+E_{r-1})P\}$ is a clique in $\Gamma(M_k(S))$.

Thus, we have proven that there exist $A \in M_k(S)$ and a permutation matrix $P \in M_k(S)$ such that $W=\{A,A+D_1P,A+(D_1+D_2)P, \ldots, A+(D_1+D_2+\ldots+D_{r-1})P\}$ is a maximal clique  in $\Gamma(M_k(S))$, where 
$D_i+D_{i+1}+\ldots+D_j$ is an invertible diagonal matrix for all $1 \leq i < j \leq r-1$. Now, define $w_i$ as the element at entry $(1,1)$ of matrix
$A+(D_1+D_2+\ldots+D_{i-1})P$ for every $i=1,2,\ldots,r$.  Since $S$ is additively cancellative and antinegative, we can conclude that $w_i \neq w_j$ for all $i \neq j \in \{1,2,\ldots,r \}$. Also, for every $1 \leq i < j \leq r$, we have $w_j=w_i+f_{i,j}$, where $f_{i,j}$ is the entry at position $(1,1)$ of matrix $D_{i+1}+D_{i+2}+\ldots+D_j$, which is an invertible diagonal matrix. Therefore $f_{i,j}$ is invertible in $S$ for every $1 \leq i < j \leq r$. This implies that $\{w_1,w_2, \ldots, w_r\}$ is a clique in $\Gamma(S)$ and thus $\omega(\Gamma(M_k(S))) \leq \omega(\Gamma(S))$, finally yielding $\omega(\Gamma(M_k(S)))=\omega(\Gamma(S))$.
\end{proof}

\bigskip

\begin{Remark}
Note that Example \ref{girth3} shows also that $\omega(\Gamma(M_k(\BB))) \geq 3$ for $k \geq 3$, while obviously $\omega(\Gamma(\BB))=2$. So, in general the clique number of a unitary Cayley graph of a matrix semiring can be larger than the clique number of the unitary Cayley graph of the underlying semiring.
\end{Remark}

\bigskip

Finally, let us examine the independence number of the unitary Cayley graph of a matrix semiring. Note that the cardinality of set $W_0$ in the proof of the next theorem can be deduced from \cite[Proposition 3.5(1)]{dolzanperm}, but we nonetheless include the proof here for the sake of completeness.

\bigskip

\begin{Theorem}
\label{independence}
Let $S$ be a finite entire antiring and $k \geq 2$. Then $\alpha(\Gamma(M_k(S))) \geq  |S|^{k^2}-\sum_{i=0}^k{{k \choose i}(-1)^i(|S|^{k-i}-1)^k} + \frac{1}{2}\left(\left(1+|S^*|^{2k}\right)^k+\left(1-|S^*|^{2k}\right)^k\right)-1$.
Moreover, if $S$ is an entire additively cancellative antiring, then $\alpha(\Gamma(M_k(S))) \leq |S|^{k^2} - k! |S^*|^k$. 
\end{Theorem}
\begin{proof}
Let $W_0$ denote the set of all matrices in $M_k(S)$ that have at least one zero row or column. Choose any $A \neq B \in W_0$. Since by Theorem \ref{invertible}, every invertible matrix in $M_k(S)$ is of the form $DP$, where $D$ is an invertible diagonal matrix and $P$ is a permutation matrix, and $S$ is antinegative, there cannot exist an edge between $A$ and $B$. Note that $|W_0| =|S|^{k^2}-|W_0'|$, where $W_0'$ is the set of all matrices with no zero rows or columns.  
Observe that there are $(|S|^{k}-1)^k$ matrices that have all rows nonzero. Now, some of them of course may have some zero columns. Suppose therefore that we have at least $i$ zero columns for some $i \in \{1,2,\ldots,k\}$. We have $k \choose i$ possible ways to choose the $i$ columns. But if we disregard the zero columns, there are $|S|^{k-i}-1$ possible ways to choose the remaining elements in every (nonzero) row.  Since there are $k$ rows, this yields $(|S|^{k-i}-1)^k$ matrices. Now, in this way we may have counted some matrices (with more than $i$ zero columns) multiple times, but the inclusion exclusion principle then yields that $|W_0'|=\sum_{i=0}^k{{k \choose i}(-1)^i(|S|^{k-i}-1)^k}$, therefore $|W_0| = |S|^{k^2}-\sum_{i=0}^k{{k \choose i}(-1)^i(|S|^{k-i}-1)^k}$.

Now, let $W_1$ denote the set of all matrices in $M_k(S)$ of the form $D_1P_1+D_2P_2$, where $D_1$ and $D_2$ are invertible diagonal matrices and $P_1$ and $P_2$ are permutation matrices such that $P_1(i) \neq P_2(i)$ for every $1 \leq i \leq k$. Again by the antinegativity of $S$, $W_0 \cup W_1$ is an independent set. Since there exist at least $k$ distinct permutations in $M_k(S)$ that satisfy the above criterion, we can easily check that $|W_1| \geq |S^*|^{2k}{k \choose 2}$. We can continue this process, by choosing $W_2$ as the set of all matrices of the form $D_1P_1+D_2P_2+D_3P_3+D_4P_4$, where $D_1,\ldots,D_4$ are invertible diagonal matrices and $P_1,\ldots,P_4$ are permutation matrices such that $P_{j_1}(i) \neq P_{j_2}(i)$ for every $1 \leq i \leq k$ and every $1 \leq j_1 \neq j_2 \leq 4$.
Observe that $W_0 \cup W_1 \cup W_2$ is still an independent set with $|W_2| \geq |S^*|^{4k}{k \choose 4}$. We can continue this process until we construct $W_{\lfloor k/2 \rfloor}$ and arrive at the independent set $W=W_0 \cup W_1 \cup \ldots \cup W_{\lfloor k/2 \rfloor}$. Note that this is a disjoint union and that $|W_1 \cup W_2 \cup \ldots \cup W_{\lfloor k/2 \rfloor}| \geq \sum\limits_{i=1}^r{|S^*|^{2ik}{k \choose 2i}}$, where $r=\lfloor k/2 \rfloor$.
But observe that $\sum\limits_{i=1}^r{|S^*|^{2ik}{k \choose 2i}}=\frac{1}{2}\left(\left(1+|S^*|^{2k}\right)^k+\left(1-|S^*|^{2k}\right)^k\right)-1$, finally yielding $\alpha(\Gamma(M_k(S))) \geq |W| \geq |S|^{k^2}-\sum_{i=0}^k{{k \choose i}(-1)^i(|S|^{k-i}-1)^k} + \frac{1}{2}\left(\left(1+|S^*|^{2k}\right)^k+\left(1-|S^*|^{2k}\right)^k\right)-1$.

Suppose now that $S$ is an additively cancellative antiring. Let $d$ denote the minimal degree of any vertex in $\Gamma(M_k(S))$. Note that for every $A \in M_k(S)$ and every invertible $U \in M_k(S)$, there is an edge between $A$ and $A+U$. Since $S$ is additively cancellative $A+U \neq A+V$ for all invertible matrices $U \neq V$, implying that $d \geq k! |S^*|^k$ by Theorem \ref{invertible}. Since $S$ is antinegative, $0 \in M_k(S)$ is a neighbour to exactly all the invertible matrices, thus $d = k! |S^*|^k$. Now, suppose $Z$ is a maximal independent set in $\Gamma(M_k(S))$. For any $v \in Z$, we know that $v$ has at least $d$ neighbours, but obviously none of them are in $Z$. Thus, $\alpha(\Gamma(M_k(S))) + d \leq |M_k(S)|=|S|^{k^2}$, thus proving our assertion.
\end{proof}

\bigskip

\begin{Example}
Let us consider the graph $\Gamma(M_2(\BB))$. Since the only units in $M_2(\BB)$ are $\left[
 \begin{matrix}
 1 & 0 \\
 0 & 1  
 \end{matrix}
 \right]$ and $\left[
 \begin{matrix}
 0 & 1 \\
 1 & 0  
 \end{matrix}
 \right]$, 
 we see that 
 $$W=\left\{ \left[
 \begin{matrix}
 0 & 0 \\
 0 & 0  
 \end{matrix}
 \right], \left[
 \begin{matrix}
 1 & 0 \\
 0 & 0  
 \end{matrix}
 \right], \left[
 \begin{matrix}
 0 & 1 \\
 0 & 0  
 \end{matrix}
 \right], \left[
 \begin{matrix}
 0 & 0 \\
 1 & 0  
 \end{matrix}
 \right], \left[
 \begin{matrix}
 0 & 0 \\
 0 & 1  
 \end{matrix}
 \right], \left[
 \begin{matrix}
 1 & 1 \\
 0 & 0  
 \end{matrix}
 \right], \left[
 \begin{matrix}
 0 & 0 \\
 1 & 1  
 \end{matrix}
 \right], \left[
 \begin{matrix}
 1 & 0 \\
 1 & 0  
 \end{matrix}
 \right], \left[
 \begin{matrix}
 0 & 1 \\
 0 & 1  
 \end{matrix}
 \right], \left[
 \begin{matrix}
 1 & 1 \\
 1 & 1  
 \end{matrix}
 \right]\right\}$$ is an independent set. Thus $\alpha((\Gamma(M_2(\BB))) \geq 10$. Since there is an edge from $\left[
 \begin{matrix}
 1 & 1 \\
 1 & 1  
 \end{matrix}
 \right]$ to all vertices in the complement of $W$, we see that any independent set that contains $v=\left[
 \begin{matrix}
 1 & 1 \\
 1 & 1  
 \end{matrix}
 \right]$ is of cardinality at most $10$. So, suppose we delete $v$ from $\Gamma(M_2(\BB))$. We are left with a graph $\Gamma'$ with $n'=15$ vertices and $e'=18$ edges, where the maximal degree of a vertex is $\Delta'=3$. By a well known bound (which is credited to Kwok in \cite{west}, but may belong to
folklore), we have $\alpha(\Gamma') \leq n' - \frac{e'}{\Delta'} = 9$, which proves that 
 $\alpha((\Gamma(M_2(\BB))) = 10$. Note that this is exactly equal to the lower bound from Theorem \ref {independence}.  
\end{Example}

\bigskip

In this paper, we have studied some graph invariants (diameter, clique number, etc.) of the Cayley graphs of semirings. The main emphasis was given to the study of matrix semirings. It might be interesting to examine how do these invariants behave under quotients, extensions and homomorphisms of these semirings. It might also be worth considering some other invariants (planarity, eigenvalues of the adjacency matrix, etc.) or to tackle the question to what degree does the Cayley graph determine the structure of the (matrix) semiring. 

\bigskip








\end{document}